\documentclass[12pt]{amsart}
\usepackage{amsmath,amssymb,amsbsy,amsfonts,latexsym,amsopn,amstext,cite,
                                               amsxtra,euscript,amscd,bm}
\usepackage{url}

\usepackage{mathrsfs}

\usepackage{color}
\usepackage[colorlinks,linkcolor=blue,anchorcolor=blue,citecolor=blue,backref=page]{hyperref}

\usepackage{graphics,epsfig}
\usepackage{graphicx}
\usepackage{float}
\usepackage{epstopdf}
\hypersetup{breaklinks=true}
\usepackage{verbatim}
\usepackage{graphicx}
\usepackage{caption}
\usepackage{subcaption}

\usepackage{soul} 
\usepackage{ulem}

\usepackage{float}
\usepackage{epstopdf}
\hypersetup{breaklinks=true}

\usepackage[np]{numprint}
\npdecimalsign{\ensuremath{.}}

\usepackage{bibentry}

\usepackage[english]{babel}
\usepackage{mathtools}
\usepackage{todonotes}
\usepackage{url}
\usepackage[norefs,nocites]{refcheck}

\begin{document}
\newtheorem{theorem}{Theorem}
\newtheorem*{thm*}{Theorem}
\newtheorem{lemma}[theorem]{Lemma}
\newtheorem{lem}[theorem]{Lemma}
\newtheorem{prop}[theorem]{Proposition}
\newtheorem{cor}[theorem]{Corollary}
\newtheorem{prob}[theorem]{Problem}
\newtheorem*{conj}{Conjecture}
\newtheorem{question}[theorem]{Question}
\newtheorem{defn}[theorem]{Definition}
\newtheorem{rem}[theorem] {Remark} 



\numberwithin{equation}{section}
\numberwithin{theorem}{section}
\numberwithin{table}{section}

\newcommand{\rad}{\operatorname{rad}}

\newcommand{\Z}{{\mathbb Z}} 
\newcommand{\Q}{{\mathbb Q}}
\newcommand{\R}{{\mathbb R}}
\newcommand{\C}{{\mathbb C}}
\newcommand{\N}{{\mathbb N}}
\newcommand{\F}{{\mathbb F}}
\newcommand{\FF}{{\mathbb F}}
\newcommand{\fq}{\mathbb{F}_q}
\newcommand{\rmk}[1]{\footnote{{\bf Comment:} #1}}

\renewcommand{\mod}{\;\operatorname{mod}}
\newcommand{\ord}{\operatorname{ord}}
\newcommand{\TT}{\mathbb{T}} 
 
\renewcommand{\i}{{\mathrm{i}}}
\renewcommand{\d}{{\mathrm{d}}}
\renewcommand{\^}{\widehat}
\newcommand{\HH}{\mathbb H}
\newcommand{\Vol}{\operatorname{vol}}
\newcommand{\area}{\operatorname{area}}
\newcommand{\dist} {\operatorname{dist}}
\newcommand{\adj} {\operatorname{adj}}

\newcommand{\tr}{\operatorname{tr}}
\newcommand{\norm}{\mathcal N} 
\newcommand{\disc}{\operatorname{disc}}

\newcommand{\intinf}{\int_{-\infty}^\infty}

\newcommand{\ave}[1]{\left\langle#1\right\rangle} 
\newcommand{\Var}{\operatorname{Var}}
\newcommand{\Prob}{\operatorname{Prob}}
\newcommand{\Cov}{\operatorname{cov}}

\newcommand{\sym}{\operatorname{Sym}}
\newcommand{\Sym}{\operatorname{Sym}}

\newcommand{\CA}{{\mathcal C}_A}
\newcommand{\cond}{\operatorname{cond}} 
\newcommand{\lcm}{\operatorname{lcm}}
\newcommand{\Kl}{\operatorname{Kl}} 
\newcommand{\leg}[2]{\left( \frac{#1}{#2} \right)}  
\newcommand{\Li}{\operatorname{Li}}

\newcommand{\sumstar}{\sideset \and^{*} \to \sum}

\newcommand{\sumf}{\sum^\flat}
 
\newcommand{\conv}{*}

\newcommand{\Ht}{\operatorname{Ht}}

\newcommand{\E}{\operatorname{\mathbb E}} 
\newcommand{\sign}{\operatorname{sign}} 
\newcommand{\meas}{\operatorname{meas}} 
\newcommand{\length}{\operatorname{length}} 

\newcommand{\GL}{\operatorname{GL}}
\newcommand{\SL}{\operatorname{SL}}
\newcommand{\Sp}{\operatorname{Sp}} 
\newcommand{\USp}{\operatorname{USp}} 

\newcommand{\re}{\operatorname{Re}}
\newcommand{\im}{\operatorname{Im}}
\newcommand{\res}{\operatorname{Res}}
  
\newcommand{\diam}{\operatorname{diam}}
 
\newcommand{\fixme}[1]{\footnote{Fixme: #1}}
 
\newcommand{\orb}{\operatorname{Orb}}
\newcommand{\supp}{\operatorname{Supp}}

\newcommand{\OP}{\operatorname{Op}} 
\newcommand{\opS}{\operatorname{S}} 
\newcommand{\opT}{\operatorname{T}} 
\newcommand{\HN}{{\cH}_{N}}
\newcommand{\OPN}{\operatorname{Op}_N} 
\newcommand{\TN}{\opT_N} 
\newcommand{\TNr}{\opT_N^{(r)}} 
\newcommand{\UNr}{U_{N,r}} 
\newcommand{\TM}{\opT_M} 
\newcommand{\Hp}{\cH_p} 
\newcommand{\OPp}{\operatorname{Op}_p} 
\newcommand{\Tp}{\operatorname{T}_p} 
 \newcommand{\diag}{\operatorname{diag}}

 \newcommand{\Fp}{\mathbb F_p}
 \newcommand{\Gal}{\operatorname{Gal}} 

 
\newcommand{\commZ}[2][]{\todo[#1,color=orange!60]{Z: #2}}
\newcommand{\commI}[2][]{\todo[#1,color=green!60]{I: #2}}
\newcommand{\commII}[2][]{\todo[#1,color=magenta!60]{I: #2}}
\newcommand{\commP}[2][]{\todo[#1,color=red!60]{P: #2}}
\newcommand{\commA}[2][]{\todo[#1,color=yellow!60]{A: #2}}
\newcommand{\commS}[2][]{\todo[#1,color=blue!60]{S: #2}}

\def\ccr#1{\textcolor{red}{#1}}
\def\cco#1{\textcolor{orange}{#1}}
\def\ccc#1{\textcolor{cyan}{#1}}

\def\Degf{e} 

\def\cont{\mathrm{cont}}

\newcommand{\bs}{\boldsymbol}
\def \a{\alpha} \def \b{\beta} \def \d{\delta} \def \e{\varepsilon} \def \g{\gamma} \def \k{\kappa} \def \l{\lambda} \def \s{\sigma} \def \t{\vartheta} \def \z{\zeta}

\newcommand{\mb}{\mathbb}

\def\sssum{\mathop{\sum\!\sum\!\sum}}
\def\ssum{\mathop{\sum\ldots \sum}}
\def\iint{\mathop{\int\ldots \int}}

\newcommand{\cc}[1]{{\color{magenta} #1}}

\newfont{\teneufm}{eufm10}
\newfont{\seveneufm}{eufm7}
\newfont{\fiveeufm}{eufm5}
%
%
\newfam\eufmfam
     \textfont\eufmfam=\teneufm
\scriptfont\eufmfam=\seveneufm
     \scriptscriptfont\eufmfam=\fiveeufm
%
%
\def\frak#1{{\fam\eufmfam\relax#1}}

\newcommand{\bflambda}{{\boldsymbol{\lambda}}}
\newcommand{\bfmu}{{\boldsymbol{\mu}}}
\newcommand{\bfxi}{{\boldsymbol{\eta}}}
\newcommand{\bfrho}{{\boldsymbol{\rho}}}

\def\eps{\varepsilon}

\def\fI{\mathfrak I}
\def\fK{\mathfrak K}
\def\fT{\mathfrak{T}}
\def\fL{\mathfrak L}
\def\fR{\mathfrak R}
\def\fQ{\mathfrak Q}

\def\fA{{\mathfrak A}}
\def\fB{{\mathfrak B}}
\def\fC{{\mathfrak C}}
\def\fM{{\mathfrak M}}
\def\fS{{\mathfrak  S}}
\def\fU{{\mathfrak U}}

\def\cPg{\cP_{\mathrm{good}}}

\def\cNg{\cN_{\mathrm{good}}}
\def\tcNg{\widetilde\cN_{\mathrm{good}}}
 
\def\sssum{\mathop{\sum\!\sum\!\sum}}
\def\ssum{\mathop{\sum\ldots \sum}}
\def\dsum{\mathop{\quad \sum \qquad \sum}}
\def\iint{\mathop{\int\ldots \int}}
 
\def\T {\mathsf {T}}
\def\Tor{\mathsf{T}_d}
\def\Tore{\widetilde{\mathrm{T}}_{d} }

\def\sM {\mathsf {M}}
\def\sL {\mathsf {L}}
\def\sK {\mathsf {K}}
\def\sP {\mathsf {P}}

\def\ss{\mathsf {s}}

\def \vk{\vec{k}}
\def \vl{\boldsymbol{\ell}}
\def \vm{\vec{m}}
\def \vn{\vec{n}}
\def \vx{\vec{x}}
\def \vf{\vec{f}}
\def \vu{\vec{u}}
\def \vv{\vec{v}}
\def \vw{\vec{w}}
\def \vz{\vec{z}}
\def\va{\vec{a}}
\def\vb{\vec{b}}

\def \balpha{\bm{\alpha}}
\def \bbeta{\bm{\beta}}
\def \bgamma{\bm{\gamma}}
\def \bdelta{\bm{\delta}}
\def \bzeta{\bm{\zeta}}
\def \blambda{\bm{\lambda}}
\def \bchi{\bm{\chi}}
\def \bphi{\bm{\varphi}}
\def \bpsi{\bm{\psi}}
\def \bxi{\bm{\xi}}
\def \bnu{\bm{\nu}}
\def \bomega{\bm{\omega}}

\def \bell{\bm{\ell}}

\def\eqref#1{(\ref{#1})}

\def\vec#1{\mathbf{#1}}

\newcommand{\abs}[1]{\left| #1 \right|}

\def\Zq{\mathbb{Z}_q}
\def\Zqx{\mathbb{Z}_q^*}
\def\Zd{\mathbb{Z}_d}
\def\Zdx{\mathbb{Z}_d^*}
\def\Zf{\mathbb{Z}_f}
\def\Zfx{\mathbb{Z}_f^*}
\def\Zp{\mathbb{Z}_p}
\def\Zpx{\mathbb{Z}_p^*}
\def\cM{\mathcal M}
\def\cE{\mathcal E}
\def\cH{\mathcal H}

\def\le{\leqslant}
\def\leq{\leqslant}
\def\ge{\geqslant}
\def\leq{\leqslant}

\def\sfB{\mathsf {B}}
\def\sfC{\mathsf {C}}
\def\sfS{\mathsf {S}}
\def\sfI{\mathsf {I}}
\def\sfT{\mathsf {T}}
\def\L{\mathsf {L}}
\def\FF{\mathsf {F}}

\def\sE {\mathscr{E}}
\def\sS {\mathscr{S}}
\def\Gal {\mathrm{Gal}}

\def\cA{{\mathcal A}}
\def\cB{{\mathcal B}}
\def\cC{{\mathcal C}}
\def\cD{{\mathcal D}}
\def\cE{{\mathcal E}}
\def\cF{{\mathcal F}}
\def\cG{{\mathcal G}}
\def\cH{{\mathcal H}}
\def\cI{{\mathcal I}}
\def\cJ{{\mathcal J}}
\def\cK{{\mathcal K}}
\def\cL{{\mathcal L}}
\def\cM{{\mathcal M}}
\def\cN{{\mathcal N}}
\def\cO{{\mathcal O}}
\def\cP{{\mathcal P}}
\def\cQ{{\mathcal Q}}
\def\cR{{\mathcal R}}
\def\cS{{\mathcal S}}
\def\cT{{\mathcal T}}
\def\cU{{\mathcal U}}
\def\cV{{\mathcal V}}
\def\cW{{\mathcal W}}
\def\cX{{\mathcal X}}
\def\cY{{\mathcal Y}}
\def\cZ{{\mathcal Z}}
\newcommand{\rmod}[1]{\: \mbox{mod} \: #1}

\def\cg{{\mathcal g}}

\def\vy{\mathbf y}
\def\vr{\mathbf r}
\def\vx{\mathbf x}
\def\va{\mathbf a}
\def\vb{\mathbf b}
\def\vc{\mathbf c}
\def\ve{\mathbf e}
\def\vf{\mathbf f}
\def\vg{\mathbf g}
\def\vh{\mathbf h}
\def\vk{\mathbf k}
\def\vm{\mathbf m}
\def\vz{\mathbf z}
\def\vu{\mathbf u}
\def\vv{\mathbf v}

\def\e{{\mathbf{\,e}}}
\def\ep{{\mathbf{\,e}}_p}
\def\eq{{\mathbf{\,e}}_q}
\def\ek{{\mathbf{\,e}}_k}

\def\Tr{{\mathrm{Tr}}}
\def\Nm{{\mathrm{Nm}}}

\def\newr{t}

 \def\SS{{\mathbf{S}}}

\def\lcm{{\mathrm{lcm}}}

 \def\0{{\mathbf{0}}}

\def\({\left(}
\def\){\right)}
\def\l|{\left|}
\def\r|{\right|}
\def\fl#1{\left\lfloor#1\right\rfloor}
\def\rf#1{\left\lceil#1\right\rceil}
\def\sumstar#1{\mathop{\sum\vphantom|^{\!\!*}\,}_{#1}}

\def\mand{\qquad \mbox{and} \qquad}

\def\tblue#1{\begin{color}{blue}{{#1}}\end{color}}

\newcommand{\fp}{\mathfrak p}

 \def \xbar{\overline x}

\newif\ifcomment

\let\varepsilon\varepsilon

\title[Quantum ergodicity for modulo prime powers]{On quantum ergodicity for higher dimensional cat maps modulo prime powers}

\author[S. Bhakta]{Subham Bhakta}
\address{School of Mathematics and Statistics, University of New South Wales, Sydney, NSW 2052, Australia.} 
\email{subham.bhakta@unsw.edu.au}

\author[I. E. Shparlinski] {Igor E. Shparlinski}
\address{School of Mathematics and Statistics, University of New South Wales, Sydney NSW 2052, Australia}
\email{igor.shparlinski@unsw.edu.au}

\begin{abstract}

A discrete model  of quantum ergodicity of   linear  maps generated by symplectic  
matrices $A \in \Sp(2d,\Z)$ modulo an integer $N\ge 1$, has been studied  for $d=1$ 
and almost all $N$ by P.~Kurlberg and Z.~Rudnick (2001).  Their  result has been strengthened  
by J.~Bourgain  (2005) and then by  A.~Ostafe, I.~E.~Shparlinski  and J.~F.~Voloch (2023). 
For arbitrary $d$ this has been studied by P.~Kurlberg, A.~Ostafe,  Z.~Rudnick
and I.~E.~Shparlinski  (2024). The corresponding equidistribution  results, for certain eigenfunctions, share the same feature: 
they apply to almost all  moduli $N$ and are unable to provide an explicit construction of 
such ``good'' values of $N$. Here, using  a bound 
of I.~E.~Shparlinski  (1978) on exponential sums with linear recurrence sequences modulo a power of a fixed prime, 
 we construct such an explicit sequence of $N$, with a   power saving 
on the discrepancy. 
\end{abstract}

\keywords{Quantum unique ergodicity, linear map, exponential sums, prime powers}
\subjclass[2020]{11L07, 81Q50}

\date{November 7, 2024}

\maketitle

\tableofcontents 

\section{Introduction}
\subsection{Quantised linear maps and discrepancy of eigenfunctions}

In what follows we freely borrow from the exposition in~\cite{KORS}. 
Namely, investigate equidistribution of eigenfunctions of 
 the quantised cat map~\cite{HB}. 
 
We need to introduce some notations. 

For an integer $N\geq 1$ we denote by $\Z_N$ the residue ring modulo $N$ 
and consider the   Hilbert space  $\HN = L^2\((\Z_N)^d\)$ 
equipped with the  scalar product
\[
\langle \varphi_1,  \varphi_2 \rangle = \frac{1}{N^d} \sum_{\vec u \in \Z_N^d}  \varphi_1(\vec u) \overline{ \varphi_2(\vec u)}, 
\qquad \varphi_1,  \varphi_2 \in \HN. 
\]
In particular, the norm of $ \varphi \in \HN$ is given by 
\[
\|\varphi\| = \langle \varphi ,  \varphi  \rangle .
\]
We then consider  the family of  unitary operators
\[
\TN(\vu):\HN\to \HN, \qquad \vu=(\vu_1,
\vu_2)\in \Z^d\times \Z^d=\Z^{2d},
\]
which are defined by the following action on $\varphi \in \HN$
\begin{equation}
\label{eq:TN}
\left(\TN(\vu)\varphi \right)(\vec w) = \e_{2N}(\vu_1\cdot \vu_2)\e_N(\vu_2\cdot \vec w) \varphi (\vec w+\vu_1),
\end{equation}
for any $\vw \in \Z_N^d$, 
where hereafter we always follow the convention that integer arguments of functions on $\Z_N$ are reduced modulo $N$ (that is, $ \varphi \(\vec w+\vu_1\) =  \varphi \(\vec w+\(\vu_1 \pmod N\)\)$).
It is also easy to verify that~\eqref{eq:TN} implies  
\[
\TN(\vu) \TN(\vv)  = \e_{2N}\left(\omega\left(\vu,\vv\right)\right)\TN(\vu+\vv),
\]
 where for  $\vec x=(\vec x_1,\vec x_2)$, 
$\vec y=(\vec y_1,\vec y_2)\in \R^d\times \R^d$ we define
\begin{equation}
\label{eq:omega}
\omega(\vec x,\vec y)  = \vec x_1 \cdot \vec y_2-\vec x_2\cdot \vec y_1,
\end{equation}
and
\[
\e(z) =  \exp\(2\pi i z\),\quad \ek(z)=\e(z/k), 
\]
see also~\cite[Equation~(2.6)]{KR2001}.

For each real-valued function $f\in C^\infty(\TT^{2d})$ (an ``observable''), where $\TT = \R/Z$ is a unit torus, 
one associates a self-adjoint operator $\OPN(f)$ on $\HN$, analogous to a pseudo-differential operator with symbol $f$, defined by 
\begin{equation}
\label{eq:OpN}
\OP_{N}(f) = \sum_{\vu\in \Z^{2d}}\^f(\vu) \TN(\vu),
\end{equation}
where
\[
f(\vec x)= \sum_{\vu\in \Z^{2d}} \^ f(\vec u)\e(\vec u \cdot \vec x).
\]

Denote by $\Sp(2d,\Z)$  the group of all integer symplectic matrices $A$  which preserve the symplectic form~\eqref{eq:omega}, that is, $\omega(A \vec x, A \vec y) = \omega(\vec x,\vec y) $.

Associated to any $A\in \Sp(2d,\Z)$ is a quantum mechanical system. 
We briefly recall the key definitions:

Assuming $A \equiv I_{2d}\pmod 2$,  where $I_{2d}$ is the $2d$-dimensional identity matrix. For each   $N\geq 1$, there is a
 unitary operator $U_N(A)$ on $\HN$ such that that for every    $f\in C^\infty(\TT^{2d})$, we have the exact Egorov property
\[
U_N(A)^*\OPN(f) U_N(A) = \OPN(f\circ A) , 
\]
where $U_N(A)^*= \overline{U_N(A)}^{\,t}$, 
we refer to~\cite{GuHa, KRR,KRDuke, KR2001, KR2, Rudnicksurvey} for  a detailed exposition in the case $d=1$ and~\cite{KelmerAnnals} for higher dimensions.  

We further assume that $A$ has an irreducible characteristic polynomial (and thus is diagonalisable over $\C$) and there are no roots of unity amongst the eigenvalues of $A$ and their nontrivial ratios.

 General results of  the Quantum Ergodicity Theorem~\cite{Schnirelman, Zelditch, CdV}, 
 make it natural to expect that  any normalised 
 sequence of eigenfunctions 
 $\psi_N\in \HN$ of  the operator $U_N(A)$  satisfy 
\[
 \lim_{N\to \infty}\left \langle \OPN(f)\psi_N,\psi_N \right \rangle = \int_{\TT^{2d}} f(\vec x)d\vec x 
\] 
for all $f\in C^\infty(\TT^{2d})$, 
   in which case we say that the sequence of eigenfunctions $\{\psi_N\}$ is uniformly distributed, 
   see also~\cite{RudnickSarnakQUE}.

To make this more quantitative, we introduce the following definition of the {\it discrepancy\/}
   \[
\Delta_A(f,N) =     \max_{\psi\in \varPsi_N(A)}  
\left|\left \langle \OPN(f)\psi, \psi\right \rangle -  \int_{\TT^{2d}}
  f(\vx)d\vx  \right| ,
\] 
where    $ \varPsi_N(A)$ is the set of all normalised (that is, with $\|\psi\|=1$)
 eigenfunctions $\psi$ of $U_N(A)$ in $\HN$.
 
\begin{rem} We note that the notion of discrepancy is 
sometimes called, especially in mathematical physics literature, 
 the {\it rate of  decay of matrix coefficients\/}. 
\end{rem}

 Then the uniformity of distribution property  for $N$ running through a certain infinite sequence 
 $\cN \subseteq \N$ means that we ask if for all $f\in C^\infty(\TT^{2d})$, 
   \begin{equation}\label{QUE for all N's}
\lim_{\substack{N\to \infty \\ N \in \cN }}  \Delta_A(f,N)   = 0 . 
\end{equation}

In turn this  leads to the following:

\begin{prob}\label{prob;a.a.N}
  Make the class of sequences $\cN$ for which~\eqref{QUE for all N's} holds as broad as possible.
\end{prob}

Problem~\ref{prob;a.a.N} has first been addressed in the work of Kurlberg and Rudnick~\cite{KR2001} where~\eqref{QUE for all N's}, 
for $d = 1$, has been established for almost all $N$, that is. when   $\cN$ is a set of asymptotic density $1$. Bourgain~\cite{BourgainGAFA} has used methods of additive combinatorics  to give a bound with a power saving 
$\Delta_A(f,N) \le N^{-\delta}$, for some unspecified $\delta > 0$ and also  for almost all $N$. 
Finally, using a  different approach via methods and results of algebraic geometry, Ostafe, Shparlinski and Voloch~\cite{OSV} have shown that  one can take any $\delta< 1/60$ in the above bound. For $d\ge 2$, the only known 
result is due to Kurlberg, Ostafe,   Rudnick and Shparlinski~\cite{KORS}, which gives~\eqref{QUE for all N's} in any dimension. We remark that although the approaches in~\cite{KR2001} and~\cite{KORS} are able to produce 
and explicit bound on the rate of convergence in~\eqref{QUE for all N's}, they are incapable of 
giving a power saving. We also note recent works~\cite{DyJe,Kim_etal,RivWolf} of somewhat different flavour.

Here concentrate on a different aspect of this question and address  the following:

\begin{prob}\label{prob;good N}
Construct an explicit sequence  $\cN$, which admit strong bounds, preferably with a power saving, on the rate of convergence in~\eqref{QUE for all N's}. 
\end{prob}

\subsection{Construction and the discrepancy bound} 
\label{sec:constr} 
Below, we always assume that $A \equiv I_{2d}\pmod 2$ and that  the characteristic polynomial  $f_A$ of  the matrix $A\in \Sp(2d,\Z)$ is irreducible over $\Z$.  In particular, $A$ is diagonalisable over $\C$. We also assume that there are no roots of unity amongst the eigenvalues of $A$ and their 
nontrivial ratios.

Let $p>2d$ be a fixed prime such that  $f_A$  
splits completely modulo $p$ and has $2d$ distinct roots (that is $p$ does not divide the
discriminant of $f_A$). We additionally assume that $p \nmid  \det A$. By the Chebotarev  Density Theorem, the set of such primes $p$ is of positive relative density in the set of all primes.

Our sequence of ``good'' moduli, required for Problem~\ref{prob;good N} is simply 
the sequence of powers 
\begin{equation}
\label{eq:Seq pk}
\cN=\{p^k:~k =0,1, \ldots\}.
\end{equation}

Our main result establishes~\eqref{QUE for all N's} for the above  sequence $\cN$ with 
a reasonably strong bound on rate of convergence. 

Let 
\begin{equation}\label{eq:kappa}
 \kappa_d =  \begin{cases}1/4,& \text{if}\ d=1,\\
1/7,& \text{if}\ d=2,\\
 \dfrac{\fl{d(2d- 5/3)}- d(d-5/3)}{2d \fl{d(2d- 5/3)+2}},& \text{if}\ d\ge 3 \ \text{and}\   d \equiv 0,1 \pmod 3,\\
 \dfrac{d}{2\rf{d(2d- 5/3)+2}},& \text{if}\ d\ge 3 \ \text{and}\   d \equiv 2 \pmod 3,
\end{cases}
\end{equation} 

We note that  for $d \ge 2$ we have 
$\kappa_d \ge 1/(4d-1)$.

\begin{theorem}\label{thm:integers}
For the sequence $\cN$ given by~\eqref{eq:Seq pk} and $N \in \cN$, 
we have 
\[ 
\Delta_A(f,N)  \le  N^{-\kappa_d+o(1)}
\] 
as $N\to \infty$.  
  \end{theorem}
  
The proof is based on a link between between $\Delta_A(f,N)$ and bounds on the 
number of solutions on certain systems of congruences, first established in~\cite{KR2001}
and then generalised and used in all other papers on this 
subject~\cite{BourgainGAFA,KORS,OSV}. In turn, we estimate the aforementioned 
number of solutions, using bounds of exponential sums with linear
recurrence sequences from~\cite{Shp}. 

We note that the cat map modulo prime powers  has also been studied by Kelmer~\cite{KelmerAHP} and Olofsson~\cite{Olof}, but their results are of different flavour.
 
 \subsection{Notation}  
 Throughout the paper, the  notations 
\[X = O(Y),  \qquad X \ll Y, \qquad Y \gg X
\] 
are all equivalent to the
statement that the inequality $|X| \le c Y$ holds with some 
constant $c> 0$, which may depend on the matrix $A$.

 We recall that the additive character with period $1$ is denoted by
\[
z \in \mathbb R\  \mapsto \ \e(z) =  \exp\(2\pi i z\).
\]
For an integer $q \ge 1$ it is also convenient to define 
\[
\eq(z) = \e(z/q).
\]

The letter $p$, with or without indices, always denotes prime numbers.

Given an algebraic number $\gamma$ we denote by $\ord(\gamma, N)$ its order modulo $N$ 
(assuming that the ideals generated by $\gamma$ and $N$ are relatively prime in an appropriate number field). In particular, for an element $\lambda \in\F_{p^s}$, $\ord(\lambda,p)$ represents the order of $\lambda$ in $\F_{p^s}$. 

Similarly, we use $\ord(A,N)$ to denote the order of $A$ modulo $N$ (which always exists 
if $\gcd(\det A, N)= 1$ and in particular for $A\in  \Sp(2d,\Z)$). 

Finally, we use $\nu_p(z)$ to denote the $p$-adic order of $z \in \Q_p$,  where $\Q_p$
is the field of  $p$-adic numbers.

\section{Operators  $\TN$ and  congruences} 
 \label{sec:TN congr}

\subsection{Preliminaries} 
\label{sec:bounds-using-tensor}

As in~\cite{KORS}, and then also in~\cite{BourgainGAFA,KORS,OSV}, we observe that 
it is enough bound  the quantity 
$\langle \TN(\vu) \psi, \psi \rangle$,
where
\[
\TN(\vu)=  \OPN(\e\( \vx\cdot \vu \)), 
\]
see also~\eqref{eq:OpN}), and $\psi \in \Psi_N(A)$ runs through  eigenfunctions 
of $U_{N}(A)$, with frequency $\vu$ growing slowly with $N$ (for example, as 
any power $N^\eta$ for any fixed $\eta> 0$).

We also use $ \ord(A,N)$ to denote the order of $A$ modulo $N$, which is always 
correctly defined if $\gcd(\det A,N) =1$, which we always assume.

For a row vector $\vu\in \Z^{2d}$, $\vu\not \equiv \0_{2d}\pmod N$, where $\0_{2d}$ is the $2d$-dimensional zero-vector, we denote by $Q_{s}(N;\vu)$ the number of solutions of the congruence 
 \begin{equation}\label{eq: Cong with-u}
\vu \left(A^{x_1}+\ldots +A^{x_{s}} -A^{y_1} -\ldots -A^{y_{s}}\right)   \equiv  \0_{2d}\pmod N,  
\end{equation}
with integers $1\leq x_i, y_i \leq  \ord(A,N) $, $i =1, \ldots, s$. 

The key inequality below connects the $2s$-th moment associated to
the basic observables $\TN(\vu)$ 
with the number of
solutions $Q_{s}(N;\vu)$ to the system~\eqref{eq: Cong with-u}. 
For even $s$, this is given (in broader generality) by~\cite[Lemma~4.1]{KORS}. 
However this parity condition is too restrictive for us, hence we show how to prove 
a result for any $s$.

\begin{lem}\label{lem: T and Q}
Let $\vu\in \Z^{2d} \setminus\{ \0_{2d}\}$. Then
\[
   \max_{\psi \in \Psi_N(A)} 
  \left| \langle \TN(\vu)\psi ,\psi  \rangle \right|^{2s}
  \leq N^d \frac{Q_{s}(N;\vu)}{\ord(A,N)^{2s} },
\]
where the maximum is taken over all normalised eigenfunctions of $U_{N}(A)$.
\end{lem}

\begin{proof}
We argue exactly as in the proof of \cite[Lemma~4.1]{KORS}. Denote $\tau=\ord(A,N)$, and consider
\[
  D(\vu) =
  \frac 1{\tau}\sum_{i=1}^\tau  \TN(\vu  A^i),\quad \mathrm{and}\quad  H(\vu)=D(\vu)^{*}D(\vu).
\]

We have
\[
| \langle \TN(\vu)\psi,\psi \rangle|^{2s} \le \| D(\vu) \|^{2s}=\| H(\vu) \|^s=\| H(\vu)^s\|,
\]
where $\| \cdot \|$ denotes the operator norm. 

At this point, our argument differs from the proof of \cite[Lemma~4.1]{KORS}. We note that  
$H(\vu)$ is not only Hermitian, but also a positive semidefinite matrix; this is because
\[
\vz^{*}H(\vu)\vz=\vz^{*}D(\vu)^{*}D(\vu)\vz=\| D(\vu)\vz\|^2.
\]

Moreover, the operator $H(\vu)$ is also unitarily diagonalisable (see \cite[Theorem~2.5.6]{HJmatrix}), with non-negative eigenvalues. This shows that $H(\vu)^s$ is also a positive semidefinite matrix. Now, it is not hard to see that $\| H(\vu)^s\|=\rho(H(\vu)^s)$,
where $\rho(H(\vu)^s)$ is the spectral radius of $H(\vu)^s$, that is, the maximum of all the eigenvalues of $H(\vu)^s$. Clearly, we then have
\[
\| H(\vu)^s\|\leq \tr(H(\vu)^s).
\]

Now the proof concludes, by the same treatment as in the proof of \cite[Lemma~4.1]{KORS}.
\end{proof} 

Next we reduce  $Q_{s}(N;\vu)$ to the number of solutions to a 
similar system of equations but without the vector $\vu$.

 \subsection{Linear independence of matrix powers} 
 \label{sec:lin indep}

For a vector $\vz = (z_1, \ldots, z_{n}) \in \R^{n}$, as usual, we  denote
\[
\| \vz \|_2 = (z_1^2 + \ldots + z_n^2)^{1/2}.
\]

We need the following result which is given by~ \cite[Lemma 4.3,~(i)]{KORS}
(in broader generality). 

Throughout this section we always assume that $A\in \Sp(2d,\Z)$ has an irreducible characteristic polynomial. 

\begin{lem}\label{lin indep lemma}
For any non-zero row vector $\vu\in \Z^{2d}$, the vectors 
\[
\vu, \vu A,  \ldots,\vu A^{2d-1}
\]
are linearly independent.
\end{lem}

We are now ready to establish the desired result which allows to remove $\vu$ in our 
considerations of  $Q_{s}(N;\vu)$.

\begin{lem}
\label{lem: no u}  
  There is a constant $C(A)$ depending only on $A$, such that if we have $p^{m+1} > C(A)\|\vu\|_2^{2d}$ for some integer $1\leq m<k$, then for any solution $\(x_1,\ldots,x_{s},y_1,\ldots, y_{s}\) \in \Z^{2s} $ to~\eqref{eq: Cong with-u} with $N =p^k$, we have
\[
    A^{x_1}+\ldots A^{x_{s}}\equiv A^{y_1}+\ldots+ A^{y_{s}}\pmod {p^{k-m}}.
\]
\end{lem}

\begin{proof}
Let us set 
\[
B=A^{x_1}+\ldots A^{x_{s}}-A^{y_1}-\ldots -A^{y_{s}}.
\]
Since $\vu B\equiv 0\pmod {p^k}$, considering the matrix $X$ whose rows are $\vu,\vu A,\ldots,\vu A^{2d-1}$ and observing that $A$ and $B$ commute,  we have $XB\equiv 0\pmod{p^k}$. In particular, multiplying both sides by  the adjoint of $X$, we get 
\begin{equation}\label{eqn:detx}
\det X\cdot B\equiv 0 \pmod {p^k}.
\end{equation}

On the other hand,  Lemma~\ref{lin indep lemma}  shows that $\det X$ is a non-zero integer. In particular, if $p^{m+1}\nmid \det X$, then the congruence~\eqref{eqn:detx} implies that $B\equiv 0 \pmod {p^{k-m}}$. The proof now follows, as we obviously have $\det X \ll \|\vu\|_2^{2d}$.
\end{proof}

Let $p$ be a split prime which does not divide the discriminant and the constant coefficient of the characteristic polynomial of $A$ (that is, exactly as  we assume in Section~\ref{sec:constr}). 

We see that we have $2d$ distinct the eigenvalues of $A$   modulo $p$, that is, in the finite field  $\F_p$ 
of $p$ elements, which using Hensel lifting give us  the roots 
 \[ 
 \lambda_{1},\ldots,\lambda_{2d}\in \Z/p^{k}\Z,
 \]
 of the characteristic polynomial of $A$  modulo $p^{k}$.

We have the following variant of~\cite[Lemma~4.4]{KORS}. 

\begin{lem}
\label{lem: no A}
    Let $p$ be any prime as in Section~\ref{sec:constr}, and let $m$ be the smallest integer 
    with $p^{m+1} > C(A)\|\vu\|_2^{2d}$ where $C(A)$ is as in Lemma~\ref {lem: no u}. 
    For any  solution $(x_1,\ldots x_{s}, y_1,\ldots, y_{s})$ to \eqref{eq: Cong with-u}
    with $N = p^k$ and $k > m$, we have
    \[\lambda_{i}^{x_1}+\ldots + \lambda_{i}^{x_{s}}\equiv \lambda_{i}^{y_1}+\ldots + \lambda_{i}^{y_{s}} \pmod {p^{k-m}}, \quad i=1,\ldots,2d.\]
\end{lem}

\begin{proof}
 By the assumption on $p$, clearly the characteristic polynomial of $A$ has $2d$ distinct roots in $\Q_p$.  In particular, $A$ is diagonalisable over $\Q_p$. Denote $\blambda_1,\ldots,\blambda_{2d} \in \Z_p$ be its eigenvalues, where $\Z_p$ is the ring of $p$-adic integers in $\Q_p$. We have
 \begin{equation}\label{eqn:bltol}
\nu_p\(\lambda_i  -  \blambda_i\) \ge k, \qquad i=1,\ldots,2d.
 \end{equation}

For each $1\leq i\leq 2d$, there exists a non-zero vector $ \vv_i\in (\Q_p)^{2d}$, for which $\vv_i A=\blambda_i \vv_i$. We scale $\vv_i = \(v_{i,1},\ldots, v_{i,2d}\)$ so that all its coordinates lie in $\Z_p$, 
with some coordinate $v_{i,j_i}$ satisfying
 \begin{equation}\label{eq:unit}
 \nu_p\(v_{i,j_i}\) = 0.
  \end{equation}
 Lemma~\ref{lem: no u} then implies that
\[
\vv_i(\blambda_{i}^{x_1}+\ldots \blambda_{i}^{x_{s}}-\blambda_{i}^{y_1}-\ldots -\blambda_{i}^{y_{s}})
 \in  \(p^{k-m}\Z_p\)^{2d}, \qquad i=1,\ldots,2d, 
\]
and thus
\[
\nu_p\(v_{i,j_i}(\blambda_{i}^{x_1}+\ldots \blambda_{i}^{x_{s}}-\blambda_{i}^{y_1}-\ldots -\blambda_{i}^{y_{s}}\) \ge k-m , \qquad i=1,\ldots,2d.
\]
The result now follows from~\eqref{eqn:bltol} and~\eqref{eq:unit}. 
\end{proof}

Hence, we see from  Lemmas~\ref{lem: no u} and~\ref{lem: no A}, that 
\begin{equation}
 \label{eq:Q < R}
Q_{s}(p^{k};\vu) \ll p^{2sm} R_{s}(p^{k-m}), 
\end{equation}
where $m$ is as in Lemma~\ref{lem: no u} and  $R_{s}(p^{r}; \vu)$ is the number of solutions to the system of 
equations 
 \begin{equation}
 \label{eq:Def R}
 \lambda_{i}^{x_1}+\ldots + \lambda_{i}^{x_{s}}\equiv \lambda_{i}^{y_1}+\ldots + \lambda_{i}^{y_{s}} \pmod {p^{r}}, \quad i=1,\ldots,2d, 
\end{equation}
in variables $x_{1}, y_{1}, \ldots, x_{s}, y_{s} =1, \ldots, \ord(A,p^k)$.

\section{Multiplicative orders and exponential sums} \label{sec:ord}
  \subsection{Multiplicative orders}  
  
 We need to collect the  simple and well-known properties of multiplicative orders 
 modulo prime powers. More general results have been given by Korobov~\cite{Kor1,Kor2}, 
 we need the following  direct consequence of~\cite[Lemma~1]{Kor1}.

 \begin{lem}\label{eq:ord growth} 
Assume that a prime  $p\ge 3$ and an integer  $\lambda\ne \pm 1$ are relatively prime. 
Let 
\[
\gamma = \nu_p\(\lambda^{\ord(\lambda, p)} -1\).
\]
Then for $k \ge \gamma$ we have 
\[
\ord(\lambda, p^k) = \ord(\lambda, p) p^{k-\gamma}.
\]
\end{lem}

Define integers $\rho_{i,j} $, $1 \le i,j \le 2d$, by 
$\rho_{i,j} \equiv  \lambda_i/\lambda_j \pmod {p^k}$. 

We also define $\gamma_i$ and $\gamma_{i,j}$ as in Lemma~\ref{eq:ord growth} 
for $\lambda_i$ and  $\rho_{i,j} $, respectively, $1 \le i,j \le 2d$

 \begin{lem}
 \label{eq:ord growth max} 
 There is a constant $c(A,p)$ depending only on $A$ and $p$, such that  
 \[
 \max_{\substack{1 \le i,j \le 2d\\ i\ne j}}\{\gamma_i, \gamma_{i,j}\} \le c(A,p). 
 \]
\end{lem}

\begin{proof} Let $\mu_1, \ldots, \mu_{2d}$ be the eigenvalues $A$ and 
$\fp$ be a prime ideal of $\Q\(\mu_1, \ldots, \mu_{2d}\)$. 
Note that $p^{\gamma_i}\mid \lambda_i^t-1$, with integer $t \ge 1$
 implies $\fp^{\gamma_i} \mid \mu_i^t-1$.  By our assumption
 on $A$ we have  $ \mu_i^t-1 \ne 0$ and since $\ord(\lambda, p) \le p-1$ 
 for any integer $\lambda \not \equiv 0 \pmod p$, we see that  $\gamma_i$ is bounded 
 only in terms of $A$ and $p$. 
 
 Similarly, $p^{\gamma_{i,j}}\mid \rho_{i,j}^t-1$ implies 
 $\fp^{\gamma_{i,j}}\mid  \mu_i^t -  \mu_j^t$ and for $i \ne j$ the same argument applies. 
\end{proof} 

 \subsection{Exponential sums}

Let $p$ be any prime as in Section~\ref{sec:constr} and let $\lambda_1, \ldots \lambda_{2d}$ be as in Section~\ref{sec:lin indep}. 

 For a vector of integers $\va = \(a_1, \ldots, a_{2d}\)$ and a positive integer $r$, we define the exponential sums 
 \[
S_r(\va) = \sum_{x=1}^{t_r}  \e_{p^r}\(a_1\lambda_1^x + \ldots + a_{2d}\lambda_{2d}^x \),
\]
where $t_r$   is the period of the sequence $a_1\lambda_1^x + \ldots + a_{2d}\lambda_{2d}^x$, $x =0,1, \ldots$, 
 modulo $p^r$.

We note that the following bound on these exponential sums is essentially established in~\cite{Shp} and in fact for 
essentially arbitrary linear recurrence sequences. Note that a similar argument has also been used in~\cite{MeShp}.

\begin{lem}\label{eq:exp sum} 
Let $p$ be any prime as in Section~\ref{sec:constr} and let $\lambda_1, \ldots \lambda_{2d}$ be as in Section~\ref{sec:lin indep}.   Then for any integer $r \ge 1$, 
uniformly over integers $a_1, \ldots, a_{2d}$  with 
\[
\gcd(a_1, \ldots, a_{2d}, p) =  1, 
 \]
we have
\[
\left| S_r(\va) \right| \le t_r^{1-1/(2d)+o(1)}, \qquad \text{as} \ r \to \infty. 
 \]
\end{lem}

\begin{proof}   The result in~\cite[Theorem~2]{Shp} is formulated  for fixed integers  
$\lambda_1, \ldots \lambda_{2d}$ (and fixed $d$ and $p$).  
Since the work~\cite{Shp} is difficult to access, we now summarise some ideas used in the proof, 
which references to much easier accessible work~\cite{MeShp}. 
In full generality,~\cite[Theorem~2]{Shp} gives the following bound 
\[
\left| \sum_{x=1}^{\tau_r} \e_{p^r} \(u(x)\) \right| \le \tau_r^{1-1/e + o(1)}
\]
on exponential sums over the full period $\tau_r$ modulo $p^r$ of an integer linear recurrence sequence $u(x)$ of order $e$, with a square-fee 
characteristic polynomial $f(X) \in \Z[X]$, such that there are no roots of unity amongst the roots of $f$ and their nontrivial ratios. This bound is based on:
\begin{itemize}
\item a polynomial representation (as polynomials in $y$) of the sequences $u(a+ \tau_s y)$ for each $a =0, \ldots, \tau_s-1$, with $s$ slowly growing with $r$, and an upper bound on $p$-adic order of at least one coefficient among every $e$ consecutive coefficients of this polynomial, 
see~\cite[Lemma~2.5]{MeShp}; 

\item  a bound on exponential sums 
 \begin{equation}\label{eq:poly exp sum}
\sum_{y=1}^{p^r} \e_{p^r} \(F(y)\)  \ll p^{r(1-1/e)}
  \end{equation}
provided $p > e$, 
with any polynomial $F(Y) \in \Z[Y]$ of the form
$F(Y) =p G(Y) + A_eY^e + \ldots + A_1Y$ for an arbitrary $G(Y) \in \Z[Y]$, 
and $\gcd(A_1, \ldots, A_e, p) = 1$, which follows, for example, from a much more general result of Cochrane  and  Zheng~\cite[Theorem~3.1]{CoZh} (with the implied constant depending only on $e$ and $p$). 
\end{itemize}

It is important to recall that the implied constant in~\eqref{eq:poly exp sum} depends only on $e$ and $p$, in particular, it does not depend on $\deg F$. 

Note that we have $e = 2d$ in our setting.

As we have mentioned a similar strategy has also been used in~\cite{MeShp}, where instead of the above complete sums, very short exponential sums are used. 

Examining the dependencies in implied constants throughout the argument of the 
proof of~\cite[Theorem~2]{Shp} one can easily verify that in fact all constants depend only on $d$, $p$ and  parameters 
$\gamma_i$ and $\gamma_{i,j}$ 
from Lemma~\ref{eq:ord growth max}, which depend  only on the matrix $A$ and the prime $p$. 
\end{proof}

\begin{rem} Certainly the parity of the number of terms in the sums $ S_r(\va)$, plays no role in argument and 
the similar statement holds for any number $e$ of terms instead of $2d$ (with the saving $1/e$). 
\end{rem}

\subsection{Bounding  $R_{2}(p^{r})$}

Here we use the idea of Kurlberg and Rudnick~\cite{KR2001} to estimate $R_{2}(p^r)$.
While it is not necessary for getting a power saving on $\Delta_A(f,p^k)$ in Theorem~\ref{thm:integers}, it leads to a larger value of $\kappa_d$.

\begin{lem}\label{eq:Bound R2} 
Let $p$ be any prime as in Section~\ref{sec:constr} and let $m$ and  $\lambda_1, \ldots \lambda_{2d}$ be as in Section~\ref{sec:lin indep}.   Then 
\[
R_{2}(p^{r}) \ll  r^2p^{7r/3}. 
\] 
\end{lem}

\begin{proof}  Since $A \in  \Sp(2d,\Z)$, we can choose an arbitrary pair of eigenvalues 
of the form $(\lambda, \lambda^{-1})$ and use only two corresponding equations from the system~\eqref{eq:Def R}. Hence, we consider the system 
\begin{align*}
&  \lambda^{x_1} +  \lambda^{x_{2}}\equiv \lambda^{y_1}+  \lambda^{y_{2}} \pmod {p^{r}},\\
&   \lambda^{-x_1} +  \lambda^{-x_{2}}\equiv \lambda^{-y_1}+  \lambda^{-y_{2}} \pmod {p^{r}}, 
\end{align*}
in variables $x_{1},x_{2}, y_{1},y_{2} =1, \ldots, \tau$, where $\tau = \ord(\lambda, p^r)$.

Denoting $ u = x_1-y_1$ and $v = x_2-y_1$ and repeating the same argument as in the 
proof of~\cite[Lemma~5]{KR2001} we derive 
\[
\(1-\lambda^u\) \(1-\lambda^v\)  \(\lambda^{u-v} +1\) \equiv 0 \pmod {p^{r}}, 
\]
see~\cite[Equation~(4.7)]{KR2001}. 

We now fix (in $\tau$ possible ways) the  value of $y_1$.

Next we fix integers $\omega_1, \omega_2, \omega_3 \ge 0$ with 
\[
\omega_1+ \omega_2+ \omega_3=r
\]
and count pairs $(u,v)$  for which the corresponding $p$-adic orders satisfy
\[ 
\nu_p  \(1-\lambda^u\) \ge  \omega_1, \quad   \nu_p \(1-\lambda^v\) \ge  \omega_2, 
\quad   \nu_p  \(\lambda^{u-v} +1\)\ge   \omega_3. 
\]

We choose tWo largest values, say $\omega_a$ and $\omega_b$, $1 \le a < b \le 3$
and note that we clearly have $\omega_a+\omega_b \ge 2r/3$.

Thus, using~\eqref{eq:ord growth}, we see that  for a fixed (in $\tau$ possible ways) value of $y_1$, the pairs $(u,v)$ take at most  $O\(p^{4r/3}\)$
 values. 

Indeed, without loss of generality we can assume, that $a=1$. Hence, for each fixed $y_1$ by Lemma~\ref{eq:ord growth}, 
there are $O\(p^{r - \omega_a}\) $  values for $u$ and hence to $x_1$. We now see that 
whether $b=2$ or $b=3$, there are $O\(p^{r -   \omega_b}\)$ values for $v$ and hence to $x_2$. 

Hence, for each choice of $\omega_1, \omega_2, \omega_3$  we have 
\[
O\(\tau p^{4r/3}\) = O\(p^{7r/3}\)
\]
choices for the triple  $(x_1, x_2, y_1)$, after which  $y_2$ is uniquely defined.

Since there are at most  $r^2$ possible choices for $\omega_1, \omega_2, \omega_3$,  the desired bound follows. 
\end{proof}

 \section{Proof of Theorem~\ref{thm:integers}}

\subsection{Bounding  $Q_{s}(N;\vu)$ via the fourth moment}   
Let $p$ be any prime as in Section~\ref{sec:constr} and let $N = p^k$. 
Assume that $k> m$  where $m$ is the smallest integer with $p^{m+1} > C(A)\|\vu\|_2^{2d}$ where $C(A)$ is as in Lemma~\ref {lem: no u}. 
Denote $T=\ord(A,N) =\ord(A,p^k)$.

Using the  orthogonality of exponential functions, it follows from~\eqref{eq:Q < R} and~\eqref{eq:Def R} that  \begin{equation}
 \begin{split}
 \label{eq:Q < exp sums}
Q_{s}&(N;\vu)\\
& \le \frac{1}{p^{2d(k-m)}}\sum_{\va \in (\Z/p^{k-m}\Z)^{2d}}\left|\sum_{x=1}^{T}\e_{p^{k-m}}\(a_1\lambda_{1}^x+\ldots+a_{2d}\lambda_{2d}^x\)\right|^{2s}.
\end{split}
\end{equation}

For each $r =0, \ldots, k-m$ we separate the contribution 
\begin{align*}
W_r & = \sum_{\substack{\va \in (\Z/p^{k-m}\Z)^{2d}\\\gcd\(a_1, \ldots, a_{2d}, p^{k-m} \)  = p^{k-m-r}}}\left|\sum_{x=1}^{T}\e_{p^{k-m}}\(a_1\lambda_{1}^x+\ldots+a_{2d}\lambda_{2d}^x\)\right|^{2s}\\
 & = \sum_{\substack{\vb \in (\Z/p^{r}\Z)^{2d}\\\gcd\(b_1, \ldots, b_{2d}, p \)  = 1}}\left|\sum_{x=1}^{T}\e_{p^{r}}\(b_1\lambda_{1}^x+\ldots+b_{2d}\lambda_{2d}^x\)\right|^{2s}\
\end{align*}
to the sum on the right hand side 
of~\eqref{eq:Q < exp sums} from vectors $\va$  for which 
$\gcd\(a_1, \ldots, a_{2d}, p^{k-m} \)  = p^{k-m-r}$.

For each $\vb \in (\Z/p^{r}\Z)^{2d}$ with $\gcd\(b_1, \ldots, b_{2d}, p \)  = 1$ we see that the period $t_r (\vb)$ of the sequence  $b_1\lambda_1^x + \ldots + b_{2d}\lambda_{2d}^x$ 
modulo $p^r$ satisfies 
\[
t_r (\vb)   \gg p^r \mand t_r (\vb)  \mid  t_{k-m} (\vb) \mid T,
\]
and hence by Lemma~\ref{eq:exp sum}  we have 
 \begin{equation}  \label{eq:Bound exp sums}
\left|\sum_{x=1}^{T}\e_{p^r}\(b_1\lambda_{1}^x+\ldots+b_{2d}\lambda_{2d}^x\)\right|
\ll T^{1+o(1)} p^{-r/(2d)}. 
\end{equation}

Therefore, assuming that 
 \begin{equation}  \label{eq:Asump1}
 s\ge 2
 \end{equation}
 and
applying~\eqref{eq:Bound exp sums} $2s-4$ times, we derive
\begin{align*}
W_r & \le \(T^{1+o(1)} p^{-r/(2d)}\)^{2s-4} \\
& \qquad \qquad \times \sum_{\substack{\vb \in (\Z/p^{r}\Z)^{2d}\\\gcd\(b_1, \ldots, b_{2d}, p \)  = 1}}\left|\sum_{x=1}^{T}\e_{p^{r}}\(b_1\lambda_{1}^x+\ldots+b_{2d}\lambda_{2d}^x\)\right|^4\\
& \le T^{2s-2+o(1)} p^{-r(s-2)/d}
 \sum_{\vb \in (\Z/p^{r}\Z)^{2d}}\left|\sum_{x=1}^{T}\e_{p^{r}}\(b_1\lambda_{1}^x+\ldots+b_{2d}\lambda_{2d}^x\)\right|^4.
\end{align*}

It is easy to see that first  by  the  orthogonality of exponential functions and then by Lemmas~\ref{eq:ord growth}   and~\ref{eq:Bound R2} we have 
\begin{align*}
 \sum_{\vb \in (\Z/p^{r}\Z)^{2d}} 
 \left|\sum_{x=1}^{T}\e_{p^{r}}\(b_1\lambda_{1}^x+\ldots+b_{2d}\lambda_{2d}^x\)\right|^4 & 
 \le (T/t_r)^4 p^{2dr} R_2(p^r)\\
 & \le k^2 T^4 p^{2dr - 5r/3+o(1)}.
\end{align*}
Hence, 
\[
W_r \le k^2 T^{2s+o(1)} p^{r\(2d -5/3  - (s-2)/d\)}.
\]

First we assume that 
 \begin{equation}
 \label{eq:cond s small1}
s-2 \le d(2d- 5/3),
\end{equation}
 and noting that $p^m \ll \|\vu\|_2^{2d}$ and $k \ll \log N \ll \log T$, we derive from~\eqref{eq:Q < exp sums} that 
\begin{align*}
Q_{s}(N;\vu) &\le \frac{1}{p^{2d(k-m)}}\sum_{r=0 }^{k-m}  W_r\\
&  \le   \frac{k^3}{p^{2d(k-m)}} T^{2s} p^{(k-m)\(2d -5/3  - (s-2)/d\)}\\ 
&\le \|\vu\|_2^{10d/3+2s-4} T^{2s+o(1)}  p^{-k\(5/3+ (s-2)/d\)}  
\end{align*}

It now follows from Lemma~\ref{lem: T and Q} that,
\begin{equation}\label{eq:small s1}
\begin{split}
\max_{\psi \in \Psi_N(A)}& \left| \langle \TN(\vu)\psi ,\psi  \rangle \right|\\
& \le  \|\vu\|_2^{(10d/3+2s-4)/(2s)}   p^{k\(5/3+ (s-2)/d\)/(4s)}T^{o(1)} \\
& \le   \|\vu\|_2^{(s-2+2d)/s} N^{-\( 5/3+ (s-2)/d- d\)/(2s)+o(1)} .
\end{split}
\end{equation}

Now we consider that case 
 \begin{equation}
 \label{eq:cond s large1}
s-2 > d(2d- 5/3).
\end{equation}
In this case we obtain 
\[
Q_{s}(N;\vu) \le  \frac{k^2}{p^{2kd}} \|\vu\|_2^{4d^2} T^{2s+o(1)}, 
\]
which  by Lemma~\ref{lem: T and Q} implies that
\begin{equation}\label{eq:large s1}
\max_{\psi \in \Psi_N(A)} \left| \langle \TN(\vu)\psi ,\psi  \rangle \right|
 \le \|\vu\|_2^{2d^2/s}  N^{-d/(2s)+o(1)}.
\end{equation}

 \subsection{Bounding  $Q_{s}(N;\vu)$ via the second moment}   
We now establish yet another bound on $Q_{s}(N;\vu)$, and thus on $ \langle \TN(\vu)\psi ,\psi  \rangle$, 
which is better than~\eqref{eq:small s1} for $d=1,2$. 

We proceed as before,  but  now we use~\eqref{eq:Bound exp sums}  $2s-2$ times and also use the orthogonality relation 
\[
 \sum_{\vb \in (\Z/p^{r}\Z)^{2d}}\ \left|\sum_{x=1}^{T}\e_{p^{r}}\(b_1\lambda_{1}^x+\ldots+b_{2d}\lambda_{2d}^x\)\right|^2
\le p^{2dr} T (T/t_r) \ll T^2 p^{r(2d-1)},
\]
instead of Lemma~\ref{eq:Bound R2}. This time, assuming that 
 \begin{equation}
 \label{eq:cond s small2}
s-1 \le d(2d-1),
\end{equation}
we derive from~\eqref{eq:Q < exp sums} that 
\begin{align*}
Q_{s}(N;\vu) &\le \frac{1}{p^{2d(k-m)}}\sum_{r=0 }^{k-m}  W_r\\
&  \le   \frac{k}{p^{2d(k-m)}} T^{2s+o(1)} p^{(k-m)\(2d -1  - (s-1)/d\)}\\ 
&\le \|\vu\|_2^{2s+2d-2} T^{2s+o(1)}  p^{-k\(1+ (s-1)/d\)} .
\end{align*}
It now follows from Lemma~\ref{lem: T and Q} that,
\begin{equation}\label{eq:small s2}
\begin{split}
\max_{\psi \in \Psi_N(A)}& \left| \langle \TN(\vu)\psi ,\psi  \rangle \right|\\
& \le  \|\vu\|_2^{(2s-2+2d)/(2s)}   p^{k\(d - 1 - (s-1)/d\)/(4s)}T^{o(1)} \\
& =  \|\vu\|_2^{(s-1+d)/s} N^{-\( (s-1)/d+1 - d\)/(2s)+o(1)} .
\end{split}
\end{equation}

We note that we do not consider the case $s > d(2d-1)$ as it never gives a 
better result, see Remark~\ref{rem: large s} below. 

\subsection{Concluding the proof}   
First employ the bound~\eqref{eq:small s1}. 
Our goal is choose $s$ with~\eqref{eq:cond s small1} which maximises  
the saving in~\eqref{eq:small s1} given by 
\[
\eta_d^-(s) = \frac{5/3+ (s-2)/d- d}{2s} = \frac{1}{2d}  - \frac{d(d-5/3) +2 }{2ds} 
\]
which is clearly add monotonically increasing function of $s$.

We choose the largest possible value of $s$, 
\[
s_1^- =   \fl{d(2d- 5/3)+2}
\]
to satisfy~\eqref{eq:Asump1}  and~\eqref{eq:cond s small1},
for which we obtain 
\[
 \eta_d^{-}(s_1^-) = \frac{1}{2d} - \frac{d(d-5/3) +2}{2d\fl{d(2d- 5/3)+2}}
=  \frac{\fl{d(2d- 5/3)}-d(d-5/3)}{2d\fl{d(2d- 5/3)+2}}. 
\]

Similarly, the saving
\[
\eta_d^+(s) =\frac{d}{2s} 
\]
 in~\eqref{eq:large s1} is   monotonically decreasing function of $s$. Hence we 
now choose the smallest possible value of $s$, 
\[
s_1^+=   \rf{d(2d- 5/3)+2}
\]
to satisfy~\eqref{eq:Asump1}  and~\eqref{eq:cond s large1},
for which we obtain 
\[
\eta_d^+(s_1^+) =  \frac{d}{2\rf{d(2d- 5/3)+2}} . 
\]

Simple calculus shows that $\eta_d^-(s_1^-) \ge  \eta_d^+(s_1^+) $ 
for $d \equiv 0,1 \pmod 3$ and  $\eta_d^-(s_1^-) < \eta_d^+(s_1^+) $
for $d \equiv 2 \pmod 3$.


Now we can 
now use~\eqref{eq:small s2} and maximise the 
corresponding saving given by 
\[
\vartheta_d(s_2) = \frac{ (s-1)/d+1 - d}{2s} = \frac{1}{2d}  - \frac{d(d-1) +1 }{2ds} , 
\]
which is clearly a monotonically increasing function of $s$.

We choose 
\[
s_2 = d(2d-1) + 1,
\]
for which we obtain 
\[
\vartheta_d(s_2) = \frac{1}{2d} - \frac{d^2 - d + 1}{2d(2d^2 - d + 1)}
=  \frac{d }{2(2d^2 - d + 1)}. 
\]

Hence, using   $\eta_d^\pm(s_1^\pm)$ for $d \ge 3$ and $\vartheta_d(s_2)$ for $d=1,2$,  we obtain 
\[
\max_{\psi \in \Psi_N(A)} \left| \langle \TN(\vu)\psi ,\psi  \rangle \right
| \le  \|\vu\|_2^{\xi_d} N^{-\kappa_d+o(1)} , 
\]
where $\kappa_d$ is given by~\eqref{eq:kappa} and 
\[
\xi_d = \max \left\{ \frac{s_1^- - 2+2d}{s_1^-},\,   \frac{s_2-1+d}{s_2}, \, \frac{2d^2}{s^{+}_1}\right\}.
\]

The proof of Theorem~\ref{thm:integers} concludes since the Fourier coefficients of the functions in $C^\infty(\TT^{2d})$ have a rapid decay (faster than any power of $\|\vu\|$). For more details, see~\cite{KORS,KRDuke}.

\begin{rem} \label{rem: large s}
To see that the case~\eqref{eq:cond s small2} is the only one to consider, 
we note that  for  $s-1> d(2d-1)$ we  have
\[
Q_{s}(N;\vu) \le \frac{1}{p^{2kd}} \|\vu\|_2^{4d^2} T^{2s+o(1)}. 
\]
Now Lemma~\ref{lem: T and Q} implies that
\[
\max_{\psi \in \Psi_N(A)} \left| \langle \TN(\vu)\psi ,\psi  \rangle \right|
 \le \|\vu\|_2^{2d^2/s}  N^{-d/(2s)+o(1)}.
 \]
One easily verifies that for  $s-1>d(2d-1)$ we have 
$ d/(2s)  \le \kappa_d$. 
\end{rem}

\section*{Acknowledgement} 

The authors are very grateful to P\"ar Kurlberg and Zeev Rudnick for 
very useful comments. 

During the preparation of this work, the authors were   supported in part by the Australian Research Council Grant DP230100530.

\end{document}